\documentclass[12pt,reqno]{amsart}
\usepackage{amsmath,amssymb}
\setbox0=\hbox{$+$}
\newdimen\plusheight
\plusheight=\ht0
\def\+{\;\lower\plusheight\hbox{$+$}\;}

\setbox0=\hbox{$-$}
\newdimen\minusheight
\minusheight=\ht0
\def\-{\;\lower\minusheight\hbox{$-$}\;}

\setbox0=\hbox{$\cdots$}
\newdimen\cdotsheight
\cdotsheight=\plusheight
\def\cds{\lower\cdotsheight\hbox{$\cdots$}}

\renewcommand{\(}{\left\(}
\renewcommand{\)}{\right\)}
\renewcommand{\[}{\left[}

 \theoremstyle{plain}
\newtheorem{theorem}{Theorem}[section]
\newtheorem{lemma}[theorem]{Lemma}

\newtheorem{corollary}[theorem]{Corollary}
\newtheorem{definition}{Definition}
\newtheorem{proposition}[theorem]{Proposition}
\newtheorem{remark}[theorem]{Remark}

\newenvironment{pf}
   {\vskip 0.15in \par\noindent{\bf Proof of Theorem}\hskip 0.5em\ignorespaces}
   {\hfill $\Box$\par\medskip}
\begin{document}
\title[Certain values of Gaussian hypergeometric series and a family of algebraic curves]
{Certain values of Gaussian hypergeometric series and a family of algebraic curves}
\author{Rupam Barman}
\address{Department of Mathematical Sciences, Tezpur University, Napaam-784028, Sonitpur, Assam, INDIA}
\email{rupamb@tezu.ernet.in}
\author{Gautam Kalita}
\address{Department of Mathematical Sciences, Tezpur University, Napaam-784028, Sonitpur, Assam, INDIA}
\email{gautamk@tezu.ernet.in}
\vspace*{0.7in}
\begin{center}

{\bf CERTAIN VALUES OF GAUSSIAN HYPERGEOMETRIC SERIES AND A FAMILY OF ALGEBRAIC CURVES}\\[5mm]

Rupam Barman and Gautam Kalita\\[.2cm]
\end{center}
\vspace{.51cm}
\noindent\textbf{Abstract:} Let $\lambda \in \mathbb{Q}\setminus \{0, -1\}$ and $l \geq 2$. Denote by
$C_{l,\lambda}$ the nonsingular projective algebraic curve over $\mathbb{Q}$ with affine equation given by
$$y^l=(x-1)(x^2+\lambda).$$
In this paper we give a relation between the number of points on $C_{l, \lambda}$
over a finite field and Gaussian hypergeometric series. We also give an alternate proof of a result of \cite{mccarthy}. 
We find some special values of ${_{3}}F_2$ and ${_{2}}F_1$ Gaussian hypergeometric series.
Finally we evaluate the value of ${_{3}}F_2(4)$ which extends a result of \cite{ono}.


\noindent{\footnotesize 2010 Mathematics Classification Numbers: 11G20, 33C20}

\section{\bf Introduction and statement of results}\label{secone}

In \cite{greene}, Greene introduced the notion of hypergeometric functions over finite fields or
\emph{Gaussian hypergeometric series} which are analogous to the classical hypergeometric series.
Gaussian hypergeometric series possess many interesting properties analogous to the classical hypergeometric series.
Finding the number of solutions of a polynomial equation over a finite field has been of interest to
mathematicians for many years. Many mathematicians have studied this problem and found interesting relations to Gaussian
hypergeometric series. For example see \cite{BK, BK2, Fuselier, koike, Lennon, mccarthy, ono, rouse, vega}.
\par
We will now restate some definitions from \cite{greene}. Let $q=p^e$ be a power of an odd prime and $\mathbb{F}_q$
the finite field of $q$ elements. Throughout this paper, $A, B, C, S, \chi, \phi, \varepsilon$ will denote complex multiplicative characters on $\mathbb{F}_q^{\times}$. 
The notation $\varepsilon, \phi$ will always be reserved
for the trivial and quadratic characters, respectively. Extend each character $\chi \in \widehat{\mathbb{F}_q^{\times}}$ to all of $\mathbb{F}_q$ by setting $\chi(0):=0$. The binomial coefficient
${A \choose B}$ is defined by
\begin{align}\label{eq0}
{A \choose B}:=\frac{B(-1)}{q}J(A,\overline{B})=\frac{B(-1)}{q}\sum_{x \in \mathbb{F}_q}A(x)\overline{B}(1-x),
\end{align}
where $J(A, B)$ denotes the usual Jacobi sum and $\overline{B}$ is the inverse of $B$.
The following special case is known from \cite{greene}
\begin{align}\label{eq01}
{A \choose \varepsilon}={A \choose A}=-\frac{1}{q}+\frac{q-1}{q}\delta(A),
\end{align}
where $\delta(A)=0$ if $A\neq\varepsilon$ and $\delta(A)=1$ if $A=\varepsilon$. With this notation, for characters $A_0, A_1,\ldots, A_n$ and $B_1, B_2,\ldots, B_n$ of $\mathbb{F}_q$,
the Gaussian hypergeometric series ${_{n+1}}F_n\left(\begin{array}{cccc}
                A_0, & A_1, & \cdots, & A_n\\
                 & B_1, & \cdots, & B_n
              \end{array}\mid x \right)$ over $\mathbb{F}_q$ is defined as
\begin{align}\label{eq00}
{_{n+1}}F_n\left(\begin{array}{cccc}
                A_0, & A_1, & \cdots, & A_n\\
                 & B_1, & \cdots, & B_n
              \end{array}\mid x \right):
              =\frac{q}{q-1}\sum_{\chi}{A_0\chi \choose \chi}{A_1\chi \choose B_1\chi}
              \cdots {A_n\chi \choose B_n\chi}\chi(x),
\end{align}
where the sum is over all characters $\chi$ of $\mathbb{F}_q$.
\par
Let $\lambda \in \mathbb{Q}\setminus \{0, -1\}$ and $l \geq 2$. Denote by
$C_{l,\lambda}$ the nonsingular projective algebraic curve over $\mathbb{Q}$ with affine equation given by
\begin{align}\label{eq1}
y^l=(x-1)(x^2+\lambda).
\end{align}
\begin{definition}
Suppose $p$ is a prime of good reduction for $C_{l, \lambda}$. Let $q=p^e$.
Define the integer $a_q(C_{l, \lambda})$ by
\begin{align}\label{eq2}
a_q(C_{l, \lambda}):=1+q-\#C_{l, \lambda}(\mathbb{F}_q),
\end{align}
where $\#C_{l, \lambda}(\mathbb{F}_q)$ denotes the number of points that the curve $C_{l, \lambda}$ has over
$\mathbb{F}_q$.
\end{definition}
It is clear that a prime $p$ not dividing $l$ is of good reduction for $C_{l, \lambda}$ if and only if
$\text{ord}_p(\lambda(\lambda+1))=0.$
\par
We now state a remark about the number of $\mathbb{F}_q$-points on $C_{l, \lambda}$.
For details, see \cite[Remark 1.1]{BK}.
\begin{remark}
Let $l \neq 3$. Then
\begin{align}\label{eq11}
\#C_{l, \lambda}(\mathbb{F}_q)=1+\#\{(x, y)\in \mathbb{F}_q^2: y^l=(x-1)(x^2+\lambda)\}.
\end{align}
Again, let $l=3$ and $p\equiv 1$ $($\emph{mod} $3)$. Then
\begin{align}\label{eq12}
\#C_{l, \lambda}(\mathbb{F}_q)=3+\#\{(x, y)\in \mathbb{F}_q^2: y^l=(x-1)(x^2+\lambda)\}.
\end{align}
\end{remark}
\begin{remark}\label{remark1}
If $l=3$, $C_{l, \lambda}$ is an elliptic curve. The change of variables $$X-Z\rightarrow X,~Y\rightarrow Y~~\text{and}~~ X\rightarrow X$$ transforms the projective curve $$C_{3, \lambda}: Y^3=(X-Z)(X^2+\lambda Z^2)$$ to
\begin{align}\label{projective-1}
Y^3&=X(X^2+2XZ+(1+\lambda)Z^2).
\end{align}
Now dehomogenizing \eqref{projective-1} by putting $X=1$ and then making the substitution $$Y\rightarrow (1+\lambda)x, Z\rightarrow (1+\lambda)y-\frac{1}{1+\lambda},$$ we find that $C_{3, \lambda}$ is isomorphic over $\mathbb{Q}$ to the elliptic curve
\begin{align}\label{200}
y^2=x^3-\frac{\lambda}{(1+\lambda)^4}.
\end{align}
\end{remark}
Ono \cite[Thm. 5]{ono}, proved that if $\lambda \in \mathbb{Q}\setminus \{0, -1\}$ and $p$ is an odd prime for which
$\text{ord}_p(\lambda(\lambda+1))=0$ then
\begin{align}\label{eq122}
{_{3}}F_2\left(\begin{array}{cccc}
                \phi, & \phi, & \phi\\
                   & \varepsilon, & \varepsilon
              \end{array}\mid \frac{1+\lambda}{\lambda} \right)
              =\frac{\phi(-\lambda)(a_p(C_{2,\lambda})^2-p)}{p^2}.
\end{align}
Note that a change of variables in Theorem 5 of \cite{ono} is required to arrive at \eqref{eq122}. In this paper, we give a proof of the following result which generalizes \eqref{eq122} to the algebraic curve $C_{l, \lambda}$ over $\mathbb{F}_q$.
\begin{theorem}\label{thm4}
Let $p$ be a prime such that \emph{ord}$_p(\lambda(\lambda+1))=0$ and $q=p^e \equiv 1$ $($\emph{mod} $l)$.
If $l\geq 2$ is such that $3 \nmid l$ or $4\nmid l$, then
\begin{align}
a_q(C_{l, \lambda})^2 &=q^2\cdot \sum_{i=1}^{l-1}\frac{J(S^{3i}, S^{-i})}{S^i(-4\lambda^3)J(S^i, S^i)}\cdot {_{3}}F_2
\left(\begin{array}{cccc}
                S^{3i}, & S^i, & S^{2i}\phi\\
                   & S^{4i}, & S^{2i}
              \end{array}\mid \frac{1+\lambda}{\lambda} \right)\notag\\
              &\hspace{.5cm}+q\cdot \sum_{i=1}^{l-1}
              \frac{\phi(-\lambda)J(S^{3i}, S^{-i})}{S^i(-4\lambda(1+\lambda)^2)J(S^i, S^i)}+Q,\notag
\end{align}
where
\begin{align}
Q=\left\{
    \begin{array}{ll}
      (l-1)(q-1)-(l-3)a_q(C_{l,\lambda}), & \hbox{if $l$ is odd;} \\
      (l-2)(q-1)-(l-2)a_q(C_{l,\lambda})\notag\\-2q\cdot \displaystyle\sum_{i=1}^{\frac{l}{2}-1}\frac{J(\phi,S^{-2i})}{J(S^{i}\phi,S^{-3i})}\cdot{_{2}}F_1\left(\begin{array}{cc}
                S^{3i}, & S^{3i}\phi\\
                   & S^{4i}
              \end{array}\mid 1+\lambda \right), & \hbox{if $l$ is even}
    \end{array}
  \right.
\end{align}
and $S$ is a character on $\mathbb{F}_q$ of order $l$.
\end{theorem}
In addition, we will also prove the following results about the number of $\mathbb{F}_q$-points on the curve $C_{l, \lambda}$.
\begin{theorem}\label{thm22}
Suppose that $q=p^e \equiv 1$ $($\emph{mod} $l)$ and \emph{ord}$_p(\lambda(\lambda+1))=0$.
If $3\nmid l$ and $\frac{q-1}{l}$ is even, then
\begin{align}\label{eq114}
-a_q(C_{l, \lambda})=q\cdot \displaystyle\sum_{i=1}^{l-1}\frac{J(\phi, S^{-i})}{J(\sqrt{S^i}, \sqrt{S^{-3i}}\phi)}
\cdot {_{2}}F_1\left(\begin{array}{cccc}
                \sqrt{S^{3i}}\phi, & \sqrt{S^{3i}}\\
                 & S^{2i}
              \end{array}\mid 1+\lambda \right)
\end{align}
and for $l=3$,
\begin{align}\label{eq115}
-a_q(C_{l, \lambda})=2+q\cdot\displaystyle\sum_{i=1}^{2}{_{2}}F_1\left(\begin{array}{cccc}
                \phi, & \varepsilon\\
                 & S^i
              \end{array}\mid 1+\lambda \right),
\end{align}
where $S$ is a character of order $l$ on $\mathbb{F}_q$.
\end{theorem}
\begin{theorem}\label{thm6}
If $q\equiv 1$ $($\emph{mod} $l)$, then for $\lambda=\frac{1}{3}$, we have
\begin{align}
-a_q(C_{l, \lambda})
=\left\{
\begin{array}{lll}
0, & \hbox{if $l\neq 3$ and $q\equiv 2$ $($\emph{mod} $3)$;}\\
q\cdot \displaystyle\sum_{i=1}^{l-1}S^i(\frac{27}{8})
\left[{\chi_3 \choose S^i}+{\chi_3^2 \choose S^i}\right], & \hbox{if $l\neq 3$ and $q\equiv 1$ $($\emph{mod} $3)$;} \\
2+q\cdot \displaystyle\sum_{i=1}^{2}
\left[{\chi_3 \choose \chi_3^i}+{\chi_3^2 \choose \chi_3^i}\right], & \hbox{if $l=3$,}
\end{array}
\right.\notag
\end{align}
where $S$ and $\chi_3$ are characters on $\mathbb{F}_q$ of order $l$ and $3$ respectively.
\end{theorem}
We also give an alternate proof of the following result of D. McCarthy.
\begin{theorem}\cite[Thm. 2.3]{mccarthy}\label{thm44}
If $q\equiv 1~($\emph{mod} $3)$, then
\begin{align}
-\frac{\phi(-2)}{q}\cdot a_q(C_{2,\frac{1}{3}})=2\emph{Re}{\chi_3 \choose \phi}\notag
\end{align}
and
\begin{align}
-\phi(-2)\cdot a_q(C_{2,\frac{1}{3}})=2\emph{Re}\left[\frac{G(\chi_3)G(\phi)}{G(\chi_3\phi)}\right],\notag
\end{align}
where $\chi_3$ is a character of order $3$ of $\mathbb{F}_q$ and $G(\chi)$ is a Gauss sum.
\end{theorem}
In section 4, we will prove the following results on special values of ${_{3}}F_2$ and ${_{2}}F_1$ hypergeometric series. 
In \cite[Thm. 1.3]{evans}, R. Evans and J. Greene gave an expression for ${_{3}}F_2(\frac{1}{4})$ which was an extension of a result of K. Ono\cite{ono}. 
The following result gives the value of ${_{3}}F_2(4)$ which also extends another result of K. Ono\cite{ono}.
\begin{theorem}\label{thm5}
If $S$ is a character on $\mathbb{F}_q$ with order not equal to $1$, $3$, or $4$, then
\begin{align}
{_{3}}F_2\left(\begin{array}{cccc}
                S^{-3}, & S^{-1}, & S^{-2}\phi\\
                   & S^{-4}, & S^{-2}
              \end{array}\mid 4 \right)
              =\left\{
              \begin{array}{ll}
\displaystyle-\frac{\phi(-3)S(16)}{q}, &\hspace{-4cm} \hbox{if $q\equiv 2$ $($\emph{mod} $3)$;}\\
\displaystyle\frac{S(-\frac{16}{27})J(S^{-1}, S^{-1})}{J(S^{-3}, S)}\left[{S \choose \chi_3}
+{S \choose \chi_3^2}\right]^2-\frac{\phi(-3)S(16)}{q},\notag\\
&\hspace{-4cm}\hbox{if $q\equiv 1$ $($\emph{mod} $3)$,}
\end{array}
\right.\notag
\end{align}
where $\chi_3$ is a character of order $3$ of $\mathbb{F}_q$.
\end{theorem}
The result of K. Ono can be obtained by putting $S=\phi$, thus solving a problem posed by M. Koike \cite[p. 465]{koike}. We remark that in view of \cite[Thm. 4.2]{greene}, there is a result similar to Theorem \ref{thm5} in which the argument $4$ is replaced by $\frac{1}{4}$. However, our result about ${_{3}}F_2(\frac{1}{4})$ will be different from the result obtained by R. Evans and J. Greene.
\begin{theorem}\label{thm8}
Let $S$ be a character on $\mathbb{F}_q$ whose order is not equal to $3$. If $S$ is square of some character on $\mathbb{F}_q$, then
\begin{align}
&(i)~{_{2}}F_1\left(\begin{array}{cccc}
                \sqrt{S^{-3}}\phi, & \sqrt{S^{-3}}\\
                 & S^{-2}
              \end{array}\mid \frac{4}{3} \right)
              =\left\{
              \begin{array}{ll}
              0, &\hspace{-2cm} \hbox{if $q\equiv 2$ $($\emph{mod} $3)$;} \\
              \displaystyle \frac{S(\frac{8}{27})J(\sqrt{S^{-1}},\sqrt{S^3}\phi)}{J(\phi,S)}
              \left[\displaystyle {S \choose \chi_3}+{S \choose \chi_3^2}\right],\\
              &\hspace{-2cm} \hbox{if $q\equiv 1$ $($\emph{mod} $3)$.}
              \end{array}
              \right.\notag\\
&(ii)~{_{2}}F_1\left(\begin{array}{cccc}
                \sqrt{S^{-3}}\phi, & \sqrt{S^{-3}}\\
                 & S^{-1}\phi
              \end{array}\mid -\frac{1}{3} \right)
              =\left\{
              \begin{array}{ll}
              0, & \hspace{-2.5cm}\hbox{if $q\equiv 2$ $($\emph{mod} $3)$;} \\
              \displaystyle\frac{S(\frac{8}{27})J(\sqrt{S^{-1}},\sqrt{S^3}\phi)}{\sqrt{S}\phi(-1)J(\phi,S)}
              \displaystyle\left[{S \choose \chi_3}+{S \choose \chi_3^2}\right],\\
              & \hspace{-2.5cm}\hbox{if $q\equiv 1$ $($\emph{mod} $3)$.}
              \end{array}
              \right.\notag\\
&(iii)~{_{2}}F_1\left(\begin{array}{cccc}
                \sqrt{S^{-3}}\phi, & \sqrt{S^{-1}}\\
                 & S^{-2}
              \end{array}\mid 4 \right)=\left\{
              \begin{array}{ll}
              0, &\hspace{-2.9cm} \hbox{if $q\equiv 2$ $($\emph{mod} $3)$;} \\
              \displaystyle\frac{\sqrt{S}(-\frac{64}{27})J(\sqrt{S^{-1}},\sqrt{S^3}\phi)}{\phi(-3)J(\phi,S)}
              \displaystyle\left[{S \choose \chi_3}+{S \choose \chi_3^2}\right],\\
              & \hspace{-2.9cm}\hbox{if $q\equiv 1$ $($\emph{mod} $3)$.}
              \end{array}
              \right.\notag\\
&(iv)~{_{2}}F_1\left(\begin{array}{cccc}
                \sqrt{S^{-3}}\phi, & \sqrt{S}\phi\\
                 & S^{-1}\phi
              \end{array}\mid \frac{1}{4} \right)=\left\{
              \begin{array}{ll}
              0, & \hspace{-2.9cm}\hbox{if $q\equiv 2$ $($\emph{mod} $3)$;} \\
              \displaystyle\frac{\sqrt{S}(-\frac{1}{27})J(\sqrt{S^{-1}},\sqrt{S^3}\phi)}{\phi(3)J(\phi,S)}
              \displaystyle\left[{S \choose \chi_3}+{S \choose \chi_3^2}\right],\\
              &\hspace{-2.9cm} \hbox{if $q\equiv 1$ $($\emph{mod} $3)$.}
              \end{array}
              \right.\notag
\end{align}
\end{theorem}

\section{\bf Preliminaries}
We start with a result which enables us to count the number of points on a curve using multiplicative characters
on $\mathbb{F}_q$ where $q=p^e$ (see \cite{vega}).
\begin{lemma}\label{lemma1}
Let $a\in\mathbb{F}_q^{\times}$. If $n|(q-1)$, then $$\#\{x\in\mathbb{F}_q: x^n=a\}=\sum \chi(a),$$
where the sum runs over all characters $\chi$ on $\mathbb{F}_q$ of order dividing $n$.
\end{lemma}
The \emph{orthogonality relations} for multiplicative characters are (see \cite[Chapter 8]{ireland}):
\begin{align}\label{lemma11}
\sum_{x\in\mathbb{F}_q}\chi(x)=\left\{
                                  \begin{array}{ll}
                                    q-1 & \hbox{if~ $\chi=\varepsilon$;} \\
                                    0 & \hbox{if ~~$\chi\neq\varepsilon$.}
                                  \end{array}
                                \right.
\end{align}
We now restate some results of R. Evans and J. Greene \cite{evans,evans1}.
The function $F(A, B; x)$ is defined by \cite{evans1}
\begin{align}\label{eq3}
F(A, B; x):=\frac{q}{q-1}
\sum_{\chi}
{ A\chi^2 \choose \chi }{A\chi \choose B\chi }\chi\left(\frac{x}{4}\right),
\end{align}
and its normalization as
\begin{align}\label{eq4}
F^*(A, B; x):=F(A, B; x)+AB(-1)\frac{\overline{A}(\frac{x}{4})}{q}.
\end{align}
Another character sum from \cite{evans} that we will need is
\begin{align}\label{eq5}
g(A, B; x):=\sum_{t\in \mathbb{F}_q}A(1-t)B(1-xt^2), ~~~~~~~x\in \mathbb{F}_q.
\end{align}
\begin{theorem}\cite[Thm. 2.2]{evans}\label{thm1}
If $A \neq C$ and $x \notin \{0, 1\}$, then
$$F^*\left(A, C; \frac{x}{x-1}\right)=\frac{A(2)A\overline{C}(1-x)}{q}\cdot g(A\overline{C}^2, \overline{A}C; 1-x).$$
\end{theorem}
\begin{theorem}\cite[Thm. 2.5]{evans}\label{thm2}
Let $C\neq \phi$, $A\notin \{\varepsilon, C, C^2\}$, $x\neq 1$. Then
\begin{align}
{_{3}}F_2\left(\begin{array}{ccc}
                A, & \overline{A}C^2, & C\phi\\
                   & C^2, & C
              \end{array}\mid x \right) &=-\frac{\overline{C}(x)\phi(1-x)}{q} \nonumber\\
              &+C(-1)A\overline{C}(4)A\overline{C}^2(1-x)
              \cdot\frac{J(\overline{A}C^2, A\overline{C})}{q^2 J(A, \overline{A}C)}
              \cdot g(A\overline{C}^2, \overline{A}C; 1-x)^2.\nonumber
              \end{align}
\end{theorem}
\begin{theorem}\cite[Thm. 1.2]{evans1}\label{thm3}
Let $R^2 \notin \{\varepsilon, C, C^2\}$. Then
$$F^*(R^2, C; x)=R(4) \frac{J(\phi, C\overline{R}^2)}{J(\overline{R}C, \overline{R}\phi)}\cdot
{_{2}}F_1\left(\begin{array}{cccc}
                R\phi, & R\\
                 & C
              \end{array}\mid x \right).$$
\end{theorem}
We now prove a result similar to the above theorem.
\begin{proposition}\label{propo1}
We have
\begin{align}
F^*(\varepsilon, C; x)=\left\{
                 \begin{array}{ll}
                   {_{2}}F_1\left(\begin{array}{cccc}
                \phi, & \varepsilon\\
                 & C
              \end{array}\mid x \right), & \hbox{if $C\neq\varepsilon$;} \\
                   -(q-2)\cdot {_{2}}F_1\left(\begin{array}{cccc}
                \phi, & \varepsilon\\
                 & C
              \end{array}\mid x \right), & \hbox{if $C=\varepsilon$.}
                 \end{array}
               \right.
\end{align}
\end{proposition}
\begin{proof} We prove the result following the technique used in \cite{evans1}. Putting $B=\varepsilon$ in the relation \cite[(4.21)]{greene}, we have
\begin{align}\label{eq119}
{\chi^2 \choose \chi}={\phi \chi \choose \chi}{\chi \choose \chi}{\phi \choose \phi }^{-1}\chi(4).
\end{align}
From \eqref{eq119} and \eqref{eq3}, we obtain
\begin{align}\label{eq111}
F(\varepsilon,C;x)&=\frac{q}{q-1}\sum_{\chi}{\chi^2 \choose \chi}{\chi \choose C\chi}\chi\left(\frac{x}{4}\right)\notag\\
&=\frac{q}{q-1}\sum_{\chi}{\chi \choose C\chi}{\phi\chi \choose \chi}{\chi \choose \chi}
{\phi \choose \phi}^{-1}\chi(4)\chi\left(\frac{x}{4}\right)\notag\\
&={\phi \choose \phi}^{-1}{_{3}}F_2\left(\begin{array}{ccc}
                \phi, & \varepsilon, & \varepsilon\\
                   & C, & \varepsilon
              \end{array}\mid x \right).
\end{align}
By \cite[Thm. 3.15 (v)]{greene}, \eqref{eq111} reduces to
\begin{align}\label{eq112}
{\phi \choose \phi}F(\varepsilon,C;x)={C \choose C}{_{2}}F_1\left(\begin{array}{cccc}
                \phi, & \varepsilon\\
                 & C
              \end{array}\mid x \right)-\frac{C(-1)}{q}{\phi \choose \varepsilon}.
\end{align}
From \eqref{eq4}, we have
\begin{align}\label{eq113}
{\phi \choose \phi}F(\varepsilon,C;x)&={\phi \choose \phi}F^*(\varepsilon,C;x)-\frac{C(-1)}{q}{\phi \choose \phi}.
\end{align}
Comparing equations \eqref{eq112} and \eqref{eq113}, we obtain
\begin{align}
F^*(\varepsilon,C;x)={C \choose C}{\phi \choose \phi}^{-1}{_{2}}F_1\left(\begin{array}{cccc}
                \phi, & \varepsilon\\
                 & C
              \end{array}\mid x \right).
\end{align}
Using \eqref{eq01}, we complete the proof of the result.
\end{proof}
The following result is due to J. Greene.
\begin{theorem}\cite[Thm. 4.4 (i) \& (ii)]{greene}\label{thm7}
For $x\in \mathbb{F}_q$,
\begin{align}
\hspace{-3cm}(i)\hspace{1cm} {_{2}}F_1\left(\begin{array}{cccc}
                A, & B\\
                 & C
              \end{array}\mid x \right)&=A(-1){_{2}}F_1\left(\begin{array}{cccc}
                A, & B\\
                 & AB\overline{C}
              \end{array}\mid 1-x \right)\notag\\
              &\hspace{.5cm}+A(-1)\displaystyle{B \choose \overline{A}C}
              \delta(1-x)-\displaystyle{B \choose C}\delta(x),\notag\\
\hspace{-3cm}(ii)\hspace{1cm} {_{2}}F_1\left(\begin{array}{cccc}
                A, & B\\
                 & C
              \end{array}\mid x \right)&=C(-1)\overline{A}(1-x){_{2}}F_1\left(\begin{array}{cccc}
                A, & C\overline{B}\\
                 & C
              \end{array}\mid \frac{x}{x-1} \right)\notag\\
              &\hspace{.5cm}+A(-1)\displaystyle{B \choose \overline{A}C}\delta(1-x),\notag
\end{align}
where $\delta$ is the function $\delta(0)=1$ and $\delta(x)=0$ for $x\neq 0$.
\end{theorem}

\section{\bf Proof of results}
\begin{pf}{\bf \ref{thm4}.}
Putting $A=S^i, B=S^i$ and $x= -\frac{1}{\lambda}$ in \eqref{eq5}, we obtain
\begin{align}
g\left(S^i, S^i; -\frac{1}{\lambda}\right)&=\sum_{t\in \mathbb{F}_q}S^{-i}(-\lambda)S^i((t-1)(t^2+\lambda))\notag
\end{align}
which gives
\begin{align}\label{eq6}
\sum_{t\in \mathbb{F}_q}S^i((t-1)(t^2+\lambda))&=S^{i}(-\lambda)g\left(S^i, S^i; -\frac{1}{\lambda}\right).
\end{align}
Moreover,
\begin{align}
&\#\{(x, y)\in \mathbb{F}_q^2: y^l=(x-1)(x^2+\lambda)\}\notag\\
&=\sum_{t\in\mathbb{F}_q}\#\{y\in\mathbb{F}_q: y^l=(t-1)(t^2+\lambda)\}\nonumber\\
&=\sum_{t\in\mathbb{F}_q, (t-1)(t^2+\lambda)\neq0}
\#\{y\in\mathbb{F}_q: y^l=(t-1)(t^2+\lambda)\}+\#\{t\in\mathbb{F}_q: (t-1)(t^2+\lambda)=0\}.\nonumber
\end{align}
Applying Lemma \ref{lemma1}, we obtain
\begin{align}
&\#\{(x, y)\in \mathbb{F}_q^2: y^l=(x-1)(x^2+\lambda)\}\nonumber\\
&=\sum_{t\in\mathbb{F}_q}\sum_{i=0}^{l-1}S^i((t-1)(t^2+\lambda))+
\#\{t\in\mathbb{F}_q: (t-1)(t^2+\lambda)=0\}\nonumber\\
&=q+\sum_{t\in\mathbb{F}_q}\sum_{i=1}^{l-1}S^i((t-1)(t^2+\lambda)).\nonumber
\end{align}
Since ord$_p(\lambda(\lambda+1))=0$, \eqref{eq11} yields
\begin{align}\label{eq13}
-a_q(C_{l, \lambda})=\sum_{i=1}^{l-1}\sum_{t\in\mathbb{F}_q}S^i((t-1)(t^2+\lambda)).
\end{align}
Squaring both sides of \eqref{eq13}, we obtain
\begin{align}
a_q(C_{l, \lambda})^2=\sum_{i=1}^{l-1}\left[\sum_{t\in\mathbb{F}_q}S^i((t-1)(t^2+\lambda))\right]^2
+\sum_{i,j=1,i\neq j}^{l-1}\sum_{t\in\mathbb{F}_q}S^{i+j}((t-1)(t^2+\lambda)).\notag
\end{align}
Again using \eqref{eq6} and \eqref{lemma11}, we deduce that
\begin{align}\label{eq7}
a_q(C_{l, \lambda})^2=\sum_{i=1}^{l-1}S^i(\lambda^2)g
\left(S^i, S^i; -\frac{1}{\lambda}\right)^2&+2(q-1)\cdot\lfloor\frac{l-1}{2}\rfloor\notag\\&+\sum_{i,j=1,i\neq j, i+j\neq l}^{l-1}\sum_{t\in\mathbb{F}_q}S^{i+j}((t-1)(t^2+\lambda)).
\end{align}
Since $3\nmid l$ or $4\nmid l$, taking $A=S^{-3i}$, $C=S^{-2i}$ and
$x=\frac{1+\lambda}{\lambda}$ in Theorem \ref{thm2}, we obtain
\begin{align}\label{eq001}
g\left(S^i, S^i; -\frac{1}{\lambda}\right)^2&=q^2\cdot\frac{S^i(-4\lambda)J(S^{-3i}, S^{i})}{J(S^{-i}, S^{-i})}
\cdot {_{3}}F_2\left(\begin{array}{cccc}
                S^{-3i}, & S^{-i}, & S^{-2i}\phi\\
                   & S^{-4i}, & S^{-2i}
              \end{array}\mid \frac{1+\lambda}{\lambda} \right)\notag\\
              &\hspace{.5cm}+q\cdot
              \frac{\phi(-\lambda)S^i(-\frac{4(1+\lambda)^2}{\lambda})J(S^{-3i}, S^{i})}{J(S^{-i}, S^{-i})}.
\end{align}
Now we find the value of $$\sum_{i,j=1,i\neq j,i+j\neq l}^{l-1}\sum_{t\in\mathbb{F}_q}S^{i+j}((t-1)(t^2+\lambda)).$$
Let $P(i_k)$ be the set of all possible values of $i$ such that $i+j\equiv k~($mod $l)$, $1\leq i,j \leq l-1$ and $i\neq j$. Then for odd values of $l$ $$\#P(i_k)=l-3$$ and for even values of $l$ $$\#P(i_k)=\left\{
                                                                         \begin{array}{ll}
                                                                           l-2, & \hbox{if $k$ is odd;} \\
                                                                           l-4, & \hbox{if $k$ is even.}
                                                                         \end{array}
                                                                       \right.$$
Therefore,
\begin{align}
&\sum_{i,j=1,i\neq j,i+j\neq l}^{l-1}\sum_{t\in\mathbb{F}_q}S^{i+j}((t-1)(t^2+\lambda))\notag\\&=\left\{
                                                                               \begin{array}{ll}
                                                  (l-3)\displaystyle\sum_{i=1}^{l-1}\sum_{t\in\mathbb{F}_q}S^i((t-1)(t^2+\lambda)), & \hbox{if $l$ is odd;} \\
                                                  (l-2)\displaystyle\sum_{i=1}^{l-1}\sum_{t\in\mathbb{F}_q}S^i((t-1)(t^2+\lambda)) - 2\sum_{i=1}^{\frac{l}{2}-1}\sum_{t\in \mathbb{F}_q}S^{2i}((t-1)(t^2+\lambda)), & \hbox{if $l$ is even.}
                                                                               \end{array}
                                                                             \right.
\end{align}
From \eqref{eq13}, \eqref{eq5} and Theorem \ref{thm1}, we deduce that
\begin{align}\label{eq002}
&\sum_{i,j=1,i\neq j,i+j\neq l}^{l-1}\sum_{t\in\mathbb{F}_q}S^{i+j}((t-1)(t^2+\lambda))\notag\\&=\left\{
                                                                                                   \begin{array}{ll}
                                                                                                    -(l-3)a_q(C_{l,\lambda}), & \hbox{if $l$ is odd;}\\
                                                                                                     -(l-2)a_q(C_{l,\lambda})-2q\cdot \displaystyle\sum_{i=1}^{\frac{l}{2}-1}\frac{J(\phi,S^{2i})}{J(S^{-i},S^{3i}\phi)}\cdot{_{2}}F_1\left(\begin{array}{cc}
                S^{-3i}\phi, & S^{-3i}\\
                   & S^{-4i}
              \end{array}\mid 1+\lambda \right), & \hbox{if $l$ is even.}
                                                                                                   \end{array}
                                                                                                 \right.
\end{align}
Using \eqref{eq001} and \eqref{eq002} in \eqref{eq7}, we complete the proof.
\end{pf}
\begin{remark}
Putting $l=2$ in Theorem \ref{thm4}, we obtain
\begin{align}\label{eq33}
a_q(C_{2,\lambda})^2=q^2\phi(-\lambda)\cdot{_{3}}F_2\left(\begin{array}{cccc}
                \phi, & \phi, & \phi\\
                   & \varepsilon, & \varepsilon
              \end{array}\mid \frac{1+\lambda}{\lambda} \right)+q,
\end{align}
which yields \eqref{eq122} over $\mathbb{F}_q$.
\end{remark}
\begin{pf}{\bf \ref{thm22}.}
Following the proof of Theorem \ref{thm4}, we obtain
\begin{align}\label{eq8}
\sum_{t\in \mathbb{F}_q}S^i((t-1)(t^2+\lambda))&=S^{i}(-\lambda)g\left(S^i, S^i; -\frac{1}{\lambda}\right)
\end{align}
and
\begin{align}\label{eq9}
\#\{(x, y)\in \mathbb{F}_q^2: y^l=(x-1)(x^2+\lambda)\}&=q+\sum_{t\in\mathbb{F}_q}\sum_{i=1}^{l-1}S^i((t-1)(t^2+\lambda)).
\end{align}
Since $S^i\neq \varepsilon$, we have $S^{-2i}\neq S^{-3i}$. Putting $A=S^{-3i},~C=S^{-2i}$
and $x=\frac{1+\lambda}{\lambda}$ in Theorem \ref{thm1}, we deduce that
\begin{align}\label{eq116}
g(S^i, S^i; -\frac{1}{\lambda})=q S^i(-\frac{8}{\lambda})F^*(S^{-3i}, S^{-2i}; 1+\lambda).
\end{align}
As $\frac{q-1}{l}$ is even, $S^i$ is a square. Also, $3\nmid l$ implies that $S^i\neq\varepsilon$.
So applying Theorem \ref{thm3}, we obtain
\begin{align}\label{eq10}
g(S^i, S^i; -\frac{1}{\lambda})=\frac{q S^i(-\frac{1}{\lambda})
J(\phi, S^{i})}{J(\sqrt{S^{-i}}, \sqrt{S^{3i}}\phi)}\cdot{_{2}}F_1\left(\begin{array}{cccc}
                \sqrt{S^{-3i}}\phi, & \sqrt{S^{-3i}}\\
                 & S^{-2i}
              \end{array}\mid 1+\lambda \right).
\end{align}
From \eqref{eq8}, \eqref{eq9}, and \eqref{eq10}, we have
\begin{align}
\#\{(x, y)\in \mathbb{F}_q^2&: y^l=(x-1)(x^2+\lambda)\}\notag\\&=q+q\cdot\sum_{i=1}^{l-1}
\frac{J(\phi, S^{i})}{J(\sqrt{S^{-i}}, \sqrt{S^{3i}}\phi)}\cdot{_{2}}F_1\left(\begin{array}{cccc}
                \sqrt{S^{-3i}}\phi, & \sqrt{S^{-3i}}\\
                 & S^{-2i}
              \end{array}\mid 1+\lambda \right).
\end{align}
Since ord$_p(\lambda(\lambda+1))=0$, \eqref{eq11} completes the proof of \eqref{eq114}.

Again for $l=3$, using Proposition \ref{propo1} in \eqref{eq116} and then combining with \eqref{eq8} and \eqref{eq9}, we obtain
\begin{align}
\#\{(x, y)\in \mathbb{F}_q^2&: y^l=(x-1)(x^2+\lambda)\}=q+
                           q\cdot\displaystyle\sum_{i=1}^{2}{_{2}}F_1\left(\begin{array}{cccc}
                \phi, & \varepsilon\\
                 & S^i
              \end{array}\mid 1+\lambda \right)\notag
\end{align}
which yields the result because of \eqref{eq12}.
\end{pf}
\begin{corollary}
Let $p$ is an odd prime for which \emph{ord}$_p(\lambda(\lambda+1))=0$. If $p\equiv 1~($\emph{mod} $3)$ and $x^2+3y^2=p$, then
\begin{align}
a_p(C_{3,-\frac{1}{2}})=\phi(2)(-1)^{x+y-1}\left(\frac{x}{3}\right)\cdot 2x\notag
\end{align}
and
\begin{align}
p\cdot\displaystyle\sum_{i=1}^{2}{_{2}}F_1\left(\begin{array}{cccc}
                \phi, & \varepsilon\\
                 & \chi_3^i
              \end{array}\mid \frac{1}{2} \right)=\phi(2)(-1)^{x+y}\left(\frac{x}{3}\right)\cdot 2x-2,\notag
\end{align}
where $\chi_3$ is a character on $\mathbb{F}_p$ of order $3$.
\end{corollary}
\begin{proof}
As mentioned in Remark \ref{remark1}, $C_{3,-\frac{1}{2}}$ is isomorphic over $\mathbb{Q}$ to the elliptic curve $y^2=x^3+2^3$, which is $2$-quadratic twist of $y^2=x^3+1$. It is known that if $E(d)$ is the $d$-quadratic twist of the elliptic curve $E$ and gcd$(p,d)=1$, then $$a_p(E)=\phi(d)a_p(E(d)).$$ Since gcd$(p,2)=1$, hence by \cite[Proposition 2]{ono}, we have $$a_p(C_{3,-\frac{1}{2}})=\phi(2)(-1)^{x+y-1}\left(\frac{x}{3}\right)\cdot 2x.$$
\par Again combining this result with the equation \eqref{eq115}, we complete the second part of the corollary.
\end{proof}
\begin{corollary}
Let $p$ be an odd prime for which \emph{ord}$_p(\lambda(\lambda+1))=0$. If $q = p^e \equiv 1~($\emph{mod} $4)$, then
\begin{align}
{_{3}}F_2\left(\begin{array}{cccc}
                \phi, & \phi, & \phi\\
                   & \varepsilon, & \varepsilon
              \end{array}\mid \frac{1+\lambda}{\lambda} \right)=\phi(\lambda){_{2}}F_1\left(\begin{array}{cccc}
                \overline{\chi_4}, & \chi_4\\
                 & \varepsilon
              \end{array}\mid 1+\lambda \right)^2-\frac{\phi(\lambda)}{q},\notag
\end{align}
where $\chi_4$ is a character of order $4$ on $\mathbb{F}_q$.
\end{corollary}
\begin{proof}
Putting $l=2$ in Theorem \ref{thm22} and then squaring both sides, we have
\begin{align}
a_q(C_{2,\lambda})^2=q^2\cdot\frac{J(\phi,\phi)^2}{J(\chi_4,\overline{\chi_4})^2}\cdot{_{2}}F_1\left(\begin{array}{cccc}
                \overline{\chi_4}, & \chi_4\\
                 & \varepsilon
              \end{array}\mid 1+\lambda \right)^2.\notag
\end{align}
Using \eqref{eq01} and then comparing with \eqref{eq33}, we complete the proof.
\end{proof}
\begin{pf}{\bf \ref{thm6}.}
Putting $\lambda=\frac{1}{3}$ in \eqref{eq1} and making the change of variables
$(x, y)\rightarrow (\frac{x}{9}+\frac{1}{3}, y)$, and then replacing $-\frac{x}{6}$ by $x$ we obtain the equivalent equation as
\begin{eqnarray*}
y^l=-\frac{8}{27}(1+x^3).
\end{eqnarray*}
Therefore,
\begin{align}
\#\{(x, y)\in \mathbb{F}_q^2: y^l=(x-1)(x^2+\frac{1}{3})\}
&=\#\{(x, y)\in \mathbb{F}_q^2: y^l=-\frac{8}{27}(1+x^3)\}\notag\\
&=q+\sum_{i=1}^{l-1}\sum_{x\in\mathbb{F}_q}S^i(-\frac{8}{27})S^i(1+x^3).\notag
\end{align}
Now recall that the binomial theorem (see \cite{greene}) for a character $A$ on $\mathbb{F}_q$
is given by
\begin{align}\label{eq121}
A(1+x)=\delta(x)+\frac{q}{q-1}\sum_\chi{A \choose \chi}\chi(x),
\end{align}
where $\delta(x)=1~($resp. $0)$
if $x=0~($resp. $x\neq 0).$ Hence
\begin{align}\label{eq120}
\#\{(x, y)\in \mathbb{F}_q^2: y^l=(x&-1)(x^2+\frac{1}{3})\}\notag\\
&=q+\sum_{i=1}^{l-1}\sum_{x\in\mathbb{F}_q}S^i(-\frac{8}{27})
\left[\delta(x^3)+\frac{q}{q-1}\sum_\chi{S^i \choose \chi}\chi(x^3)\right]\notag\\
&=q+\sum_{i=1}^{l-1}S^i(-\frac{8}{27})+\frac{q}{q-1}\sum_{i=0}^{l-1}S^i(-\frac{8}{27})
\sum_\chi{S^i \choose \chi}\sum_{x\in\mathbb{F}_q}\chi^3(x).
\end{align}
By \eqref{lemma11}, $\sum_{x\in\mathbb{F}_q}\chi^3(x)$ is nonzero if $\chi^3=\varepsilon$,
which is possible only for $\varepsilon$, $\chi_3$ and $\chi_3^2$. Therefore, \eqref{eq120} reduces to
\begin{align}
\#\{(x, y)\in \mathbb{F}_q^2: y^l=(x&-1)(x^2+\frac{1}{3})\}\notag\\
&=\left\{
\begin{array}{ll}
q, & \hbox{if $q\equiv 2$ $($mod $3)$;} \\
q+q\cdot\displaystyle\sum_{i=1}^{l-1}S^i(\frac{27}{8})
\left[{\chi_3 \choose S^i}+{\chi_3^2 \choose S^i}\right], & \hbox{if $q\equiv 1$ $($mod $3)$,}
\end{array}
\right.\notag
\end{align}
which completes the proof of the result because of \eqref{eq11} and \eqref{eq12}.
\end{pf}
We now give an alternate proof of a result of McCarthy.
\begin{pf}{\bf \ref{thm44}.}
Since $q\equiv 1~($mod $3)$, putting $l=2$ in Theorem \ref{thm6} we find that
\begin{align}
-a_q(C_{2,\frac{1}{3}})=q\phi(6)\left[{\chi_3 \choose \phi}+{\chi_3^2 \choose \phi}\right].\notag
\end{align}
Since $\phi(-3)=1$ if and only if $q\equiv 1~($mod $3)$ and
$\overline{{\chi_3 \choose \phi}}={\chi_3^2 \choose \phi}$, the first part follows.

Again the second part follows from the fact that if $\chi\psi$ is nontrivial, then
\begin{align}
J(\chi, \psi)=\frac{G(\chi)G(\psi)}{G(\chi\psi)},\notag
\end{align}
where $J(\chi, \psi)$ and $G(\chi)$ are Jacobi and Gauss sums respectively.
\end{pf}
Simplifying the expression for $a_q(C_{l,\lambda})$ given in Theorem \ref{thm6},
we obtain the following result which generalizes the case $l=2$ treated in Theorem \ref{thm44}.
\begin{corollary}
Let $d=$ \emph{lcm}$(3, l)$. If $q\equiv 1$ $($\emph{mod} $d)$, then
\begin{align}\label{coro1}
-a_q(C_{l, \frac{1}{3}})=
\left\{
\begin{array}{lll}
2+2q\cdot\emph{Re}\displaystyle\left[{\chi_3 \choose \chi_3}+{\chi_3^2 \choose \chi_3}\right],
\hspace{3.27cm} \hbox{if $l=3$;}\\
2q\cdot \displaystyle\sum_{i=1}^{\frac{l-1}{2}}\emph{Re}
\left[S^i\left(\frac{27}{8}\right)\left\{{\chi_3 \choose S^i}+{\chi_3^2 \choose S^i}\right\}\right],
\hspace{1cm}\hbox{if $l$ is odd, $l>3$;} \\
2q\cdot \displaystyle\left[\phi(-2)\emph{Re}{\chi_3 \choose \phi}
+\displaystyle\sum_{i=1}^{\frac{l-2}{2}}\emph{Re}
\left\{S^i\left(\frac{27}{8}\right)
\left({\chi_3 \choose S^i}+{\chi_3^2 \choose S^i}\right)\right\}\right],\\ \hspace{8.5cm}  \hbox{if $l$ is even;}
                           \end{array}
                         \right.\notag
\end{align}
where $S$ and $\chi_3$ are characters on $\mathbb{F}_q$ of order $l$ and $3$ respectively.
\end{corollary}

\section{Values of ${_{3}}F_2$ and ${_{2}}F_1$ Gaussian hypergeometric series}
In this section, we will give the proof of Theorem \ref{thm5} and Theorem \ref{thm8}. We now prove the following lemmas from which Theorem \ref{thm5} and Theorem \ref{thm8} follow directly.
\begin{lemma}\label{lemma2}
If $S$ is a character on $\mathbb{F}_q$ whose order is not equal to $1$, $3$ or $4$, then
\begin{align}
{_{3}}F_2\left(\begin{array}{cccc}
                S^{-3}, & S^{-1}, & S^{-2}\phi\\
                   & S^{-4}, & S^{-2}
              \end{array}\mid \frac{1+\lambda}{\lambda} \right)&
              =\frac{J(S^{-1}, S^{-1})}{q^2 S(-4\lambda^3) J(S^{-3}, S)}
              \left[\sum_{x\in\mathbb{F}_q}S((x-1)(x^2+\lambda))\right]^2\notag\\
              &-\frac{S^2(\frac{1+\lambda}{\lambda})}{q}\phi(-\lambda).\notag
\end{align}
\end{lemma}
\begin{proof}
Since $S$ is a character on $\mathbb{F}_q$ whose order is not equal to $1$, $3$ or $4$, so applying
Theorem \ref{thm2} directly for $A=S^{-3}$, $C=S^{-2}$ and $x=\frac{1+\lambda}{\lambda}$, we get
\begin{align}
{_{3}}F_2\left(\begin{array}{cccc}
                S^{-3}, & S^{-1}, & S^{-2}\phi\\
                   & S^{-4}, & S^{-2}
              \end{array}\mid \frac{1+\lambda}{\lambda} \right)
              &=\frac{J(S^{-1}, S^{-1})}{q^2 S(-4\lambda) J(S^{-3}, S)}g(S, S;-\frac{1}{\lambda})^2\notag\\
              &-\frac{S^2(\frac{1+\lambda}{\lambda})}{q}\phi(-\lambda),\notag
\end{align}
which yields the result because of
\begin{align}
g(S, S; -\frac{1}{\lambda})=\sum_{x\in \mathbb{F}_q}S^{-1}(-\lambda)S((x-1)(x^2+\lambda)).\notag
\end{align}
\end{proof}
\begin{lemma}\label{lemma3}
Let $S$ be any character on $\mathbb{F}_q$. For $\lambda=\frac{1}{3}$, we have
\begin{align}
\sum_{x\in\mathbb{F}_q}S((x-1)(x^2+\lambda))
=\left\{
\begin{array}{ll}
0, & \hbox{if $q\equiv 2$ $($\emph{mod} $3)$;} \\
qS(-\frac{8}{27})\displaystyle\left[{S \choose \chi_3}+
{S \choose \chi_3^2}\right], & \hbox{if $q\equiv 1$ $($\emph{mod} $3)$,}
\end{array}
\right.\notag
\end{align}
where $\chi_3$ is a character of order $3$ on $\mathbb{F}_q$.
\end{lemma}
\begin{proof}
As shown in the proof of Theorem \ref{thm6}, $$y^l=(x-1)(x^2+\frac{1}{3})$$
is equivalent to $$y^l=-\frac{8}{27}(1+x^3).$$
Hence using \eqref{eq121}, we have
\begin{align}
\sum_{x\in\mathbb{F}_q}S((x-1)(x^2+\frac{1}{3}))&=\sum_{x\in\mathbb{F}_q}S(-\frac{8}{27})S(1+x^3)\notag\\
&=S(-\frac{8}{27}) + \frac{q}{q-1}S(-\frac{8}{27})\sum_{x\in\mathbb{F}_q}\sum_{\chi}{S \choose \chi}\chi^3(x).\notag
\end{align}
Following the proof of Theorem \ref{thm6}, the result follows easily.
\end{proof}
\begin{lemma}\label{lemma4}
If $S$ is square of some character on $\mathbb{F}_q$ and $S$ is not of order $3$, then
\begin{align}
\sum_{x\in\mathbb{F}_q}S((x-1)(x^2+\lambda))
=\frac{qJ(\phi, S)}{J(\sqrt{S^{-1}}, \sqrt{S^3}\phi)}\cdot {_{2}}F_1\left(\begin{array}{cccc}
                \sqrt{S^{-3}}\phi, & \sqrt{S^{-3}}\\
                 & S^{-2}
              \end{array}\mid 1+\lambda \right)\notag
\end{align}
\end{lemma}
\begin{proof}
We have
\begin{align}
\sum_{x\in \mathbb{F}_q}S((x-1)(x^2+\lambda))=S(-\lambda)g(S, S; -\frac{1}{\lambda}).\notag
\end{align}
Since $S$ is a square of some character of $\mathbb{F}_q$, applying Theorem \ref{thm1}, we obtain
\begin{equation}\label{eq118}
\sum_{x\in \mathbb{F}_q}S((x-1)(x^2+\lambda))=qS^3(2)F^*(S^{-3}, S^{-2}; 1+\lambda).
\end{equation}
Also $S$ is not of order $3$. Using Theorem \ref{thm3} we complete the proof.
\end{proof}
Following the proof of Lemma \ref{lemma4} and applying Proposition \ref{propo1} in \eqref{eq118},
we have the following result.
\begin{lemma}
If $S$ is a character of order $3$ on $\mathbb{F}_q$, then
\begin{align}
\sum_{x\in\mathbb{F}_q}S((x-1)(x^2+\lambda))=q\cdot{_{2}}F_1\left(\begin{array}{cccc}
                \phi, & \varepsilon\\
                 & S
              \end{array}\mid x \right).
\end{align}
\end{lemma}
\begin{pf}{\bf \ref{thm5}.}
Putting $\lambda=\frac{1}{3}$ in Lemma \ref{lemma2} and then combining it with Lemma \ref{lemma3} we complete the proof.
\end{pf}
\begin{pf}{\bf \ref{thm8}.}
(i) Putting $\lambda=\frac{1}{3}$ in Lemma \ref{lemma4} and then using Lemma \ref{lemma3}, we complete the proof.\\

(ii) Taking $x=\frac{4}{3}$ in Theorem \ref{thm7} (i), we obtain
\begin{align}
{_{2}}F_1\left(\begin{array}{cccc}
                \sqrt{S^{-3}}\phi, & \sqrt{S^{-3}}\\
                 & S^{-1}\phi
              \end{array}\mid -\frac{1}{3} \right)=\sqrt{S}\phi(-1){_{2}}F_1\left(\begin{array}{cccc}
                \sqrt{S^{-3}}\phi, & \sqrt{S^{-3}}\\
                 & S^{-2}
              \end{array}\mid \frac{4}{3} \right).\notag
\end{align}
Now using (i), we complete the proof.\\

(iii) Applying Theorem \ref{thm7} (ii) for $x=\frac{4}{3}$, we have
\begin{align}
{_{2}}F_1\left(\begin{array}{cccc}
                \sqrt{S^{-3}}\phi, & \sqrt{S^{-1}}\\
                 & S^{-2}
              \end{array}\mid 4 \right)=\sqrt{S^3}\phi(-3){_{2}}F_1\left(\begin{array}{cccc}
                \sqrt{S^{-3}}\phi, & \sqrt{S^{-3}}\\
                 & S^{-2}
              \end{array}\mid \frac{4}{3} \right)\notag
\end{align}
and the result follows from (i).\\

(iv) Using Theorem \ref{thm7} (ii) for $x=-\frac{1}{3}$, we find that
\begin{align}
{_{2}}F_1\left(\begin{array}{cccc}
                \sqrt{S^{-3}}\phi, & \sqrt{S}\phi\\
                 & S^{-1}\phi
              \end{array}\mid \frac{1}{4} \right)=\phi(-1)\sqrt{S^3}\phi(\frac{3}{4}){_{2}}F_1\left(\begin{array}{cccc}
                \sqrt{S^{-3}}\phi, & \sqrt{S^{-3}}\\
                 & S^{-1}\phi
              \end{array}\mid -\frac{1}{3} \right)\notag
\end{align}
and then proof follows from the proof of (ii).
\end{pf}

\section*{Acknowledgment}
We thank Ken Ono for many helpful suggestions during the preparation of the article. 
The second author acknowledges the financial support of Department of Science and Technology, 
Government of India for supporting a part of this work under INSPIRE Fellowship.

\end{document}